\documentclass[preprint,12pt]{elsarticle}

\usepackage[english]{babel} 
\usepackage[utf8]{inputenc}
\usepackage{amsmath,amssymb,amsthm,amscd}
\usepackage{graphicx}
\usepackage{fancyhdr}
\usepackage{enumerate}
\usepackage{indentfirst}
\usepackage{geometry}
\usepackage{hyperref}
\usepackage[all]{xy}
\usepackage{calc,color}

\usepackage{tikz}
\usepackage{pgfplots}
\usetikzlibrary{decorations.markings}
\tikzset{->-/.style={decoration={
  markings,
  mark=at position #1 with {\arrow{>}}},postaction={decorate}}}

\usepackage{tkz-euclide}
\usetikzlibrary{intersections}
\newcommand{\R}{\ensuremath{\mathbb{R}}}

\newcommand{\delim}[3]{\left#1 #3 \right#2}  
\newcommand{\pin}[1]{\delim{\langle}{\rangle}{#1}}

\usepackage[pagewise]{lineno}
\newtheorem{theorem}{Theorem}[section]

\newtheorem{remark}[theorem]{Remark} 
 
\numberwithin{equation}{section}

\journal{Bulletin of the Brazilian Mathematical Society, New Series}

\begin{document}

\title{Poincaré compactification for semiflows of reaction-diffusion equations with large diffusion and convection heating at the boundary}

\author{L. Pires}
\address{State University of Ponta Grossa,  Ponta Grossa PR, Brazil}
\ead{lpires@uepg.br}

%\linenumbers
\begin{abstract}
In this paper, we study the Poincaré compactification of the limiting planar semiflow of a coupled PDE-ODE system composed by a reaction-diffusion equation with large diffusion coupled with an ODE by a boundary condition in a heating transition region. The nonlinear sources are dissipative polynomials. We guarantee conditions to apply the Invariant Manifold Theorem in order to reduce the dimension of the PDE and we prove that the compactified vector fields are close in the $C^1$-norm. 
\end{abstract}

\maketitle

\thispagestyle{empty}

%\noindent \hrulefill \par
  \par
%\noindent \hrulefill 
%\tableofcontents
%\noindent \hrulefill 

\normalsize \selectfont

\section{Introduction}

Reaction-diffusion equations with large diffusion and convection heating at the boundary form a family of PDE-ODE systems parameterized by $\varepsilon$, where the coupling occurs through part of the boundary and vector fields. As $\varepsilon \to 0$, the limiting problem reduces to a planar ODE. In \cite{Pires_2023_4}, it was demonstrated that the semiflow of the PDE-ODE system is a $C^1$ perturbation of the limiting planar semiflow. 

The technique employed in \cite{Pires_2023_4} relies on the existence of an invariant manifold, which reduces the problem to a two-dimensional space. Subsequently, the existence of a global attractor enables a comparison of the semiflows within a compact set that captures all the dynamics of the problem. 
%In this framework, the semiflows exhibit behavior characteristic of Morse-Smale semiflows on a two-dimensional compact manifold, where structural stability is well-established by the results in \cite{Palis_Geometric}.
Moreover, we can cut of the polinomial sources in such a way that it becomes globally Lipschitz without change the dynamics. We use the invariant manifold technique to reduce the PDE-ODE system to a two-dimensional space identified with $\R^2$. Since the reduced problem acts on the plane, we perform a Poincaré compactification to compare the semiflows. In the Poincaré disk, we prove the $C^1$ closeness of the perturbed semiflow $(\varepsilon>0)$ and unperturbed semiflow $(\varepsilon=0)$.

Let $\Omega$ be a cylindrical solid body placed along the $x$-axis between $x=0$ and $x=1$ and, let $T$ be a cylinder-shaped tank of the same radius as $\Omega$ along the $x$-axis between $x=-1$ and $x=0$ (Figure \ref{domain_draw}). The area of contact between $\Omega$ and $T$ at $x=0$ is denoted by $\Gamma_0$. The tank $T$ is filled with liquid well stirred throughout the process. At $x=-1$ through one pipe, a liquid is entering the tank at a constant flow rate, and, through another pipe, a liquid of temperature $v(t)$ is leaving the tank at the same entering flow rate. 
\begin{figure}[h]
\begin{center}
\begin{tikzpicture}[line cap=round,line join=round]
\tkzDefPoint(3,0.8){A}
\tkzDefPoint(3,-0.8){B}
\tkzDefPoint(3,0){0}
\tkzDefPoint(5.7,0.5){P}
\tkzDefPoint(4.5,0){x}
\tkzDefPoint(6,0){y}
\tkzDefPoint(9,0){1}
\tkzDefPoint(9,0.8){E}
\tkzDefPoint(9,-0.8){F}
\tkzDefPoint(3,-0.45){A1}
\tkzDefPoint(2,-0.45){A2}
\tkzDefPoint(3,-0.55){A3}
\tkzDefPoint(2,-0.55){A4}
\tkzDrawSegments[dashed](A1,A2)
\tkzDrawSegments[dashed](A3,A4)
\tkzDefPoint(3,0.45){A5}
\tkzDefPoint(2,0.45){A6}
\tkzDefPoint(3,0.55){A7}
\tkzDefPoint(2,0.55){A8}
\tkzDrawSegments[dashed](A5,A6)
\tkzDrawSegments[dashed](A7,A8)
\draw[dashed,thin,name path=elip] (0) ellipse (0.4cm and 0.8cm);
%\draw[dashed,thin,name path=elip] (x) ellipse (0.4cm and 0.8cm);
\draw[dashed,thin,name path=elip] (y) ellipse (0.4cm and 0.8cm);
\draw[dashed,thin,name path=elip] (1) ellipse (0.4cm and 0.8cm);
\draw[draw=none,name path=0C] (0)--(1);
%\tkzDrawSegments[->](x,P)
\tkzDrawPoints[fill=black](0,y,1)
\tkzDrawSegments[dashed](0,1)
\tkzDrawSegments[dashed,thick](A,E)
\tkzDrawSegments[dashed,thick](B,F)
\draw (6.1,0.4) node{\tiny $\Gamma_0$};
%\draw (4.5,0.4) node{\tiny $\Gamma_0$};
\draw (2.9,-0.3) node{\tiny $-1$};
%\draw (4.5,-0.3) node{\tiny $0$};
\draw (6,-0.3) node{\tiny $0$};
\draw (9,-0.3) node{\tiny $1$};
\draw (7.5,-0.5) node{\tiny $\Omega$};
%\draw (5.3,-0.5) node{\tiny $\Omega_0$};
\draw (4.5,-0.5) node{\tiny $T$};
\draw (4.5,-1) node{\tiny liquid};
\draw (7.5,-1) node{\tiny solid};
%\draw (4.5,0.5) node{\tiny $\Gamma_0$};
\draw (7.5,0.5) node{\tiny $u^\varepsilon(x,t)$};
%\draw (5.3,0.5) node{\tiny $w(t)$};
\draw (4.5,0.5) node{\tiny $v(t)$};
\end{tikzpicture}
\caption{Domain}\label{domain_draw}
\end{center}
\end{figure}

This type of problem was considered in \cite{Axioms23} where it is well-known that the symmetry of the cylindrical solid domain allows us to consider the unknown temperature $u^\varepsilon(x,t)$ inside $\Omega$ depends on one spatial variable $x\in(0,1)$. Thus, the  coupled system to describe the heat transfer through $\Gamma_0$ is given by the following equations:
\begin{equation}\label{perturbed_problem}
\begin{cases}
u_t^\varepsilon-(a_\varepsilon(x)u_x^\varepsilon)_x+\lambda u^\varepsilon=f_1(u^\varepsilon),\quad &\hbox{in }(0,1)\times (0,\infty),\\
u^\varepsilon_x(1,t)=0, & \hbox{in }(0,\infty),\\
-a_\varepsilon(0)u^\varepsilon_x(0,t)=g_1(v), & \hbox{in } (0,\infty),\\
\dot{v}+\beta v=f_2(u^\varepsilon(0,t))+g_2(v), & \hbox{in }(0,\infty),\\
u^\varepsilon(x,0)=u_0, & \hbox{in }(0,1),\\
v(0)=v_0, 
\end{cases}
\end{equation}
where $\varepsilon\in (0,\varepsilon_0]$,  is a parameter, $\lambda$ and $\beta$ are positive constants and the nonlinerities $f_1,f_2,g_1$ and $g_2$ are polynomials such that, the maximum $d$ between the degrees of $f_1,f_2,g_1,g_2$ satisfies $d-2\geq 0$ and $d-1$ is greater than or equal to degree of $g_1$ and $f_1$. Moreover, we assume the following dissipative condition:
\begin{equation}\label{f_assumption}
\limsup_{|x|\to \infty}\frac{h(x)}{x} < 0,
\end{equation}
where $h=f_1,f_2,g_1,g_2$.

As in \cite{Pires_2023_4} we assume that the diffusion coefficients $a_\varepsilon\in C^1([0,1],[m_0,\infty))$, for some $m_0>0$, satisfy 
\begin{equation}\label{large_diffusion}
\tau(\varepsilon):=\min_{x\in[0,1]} a_\varepsilon(x)\to \infty\quad\hbox{as}\quad \varepsilon\to 0.
\end{equation}
The convergence \eqref{large_diffusion} is known as large diffusion and was widely studied in several works including \cite{Lpires1, Carvalho, Hale1986,hale_rocha_87, Pires_Sam_22}. In order to describe our results and provide motivation for the limiting equation $\varepsilon\to 0$, we integrate the PDE in \eqref{perturbed_problem} to obtain 
$$
\int_0^1 u^\varepsilon_t\,dx-[a_\varepsilon(1)u^\varepsilon_x(1,t)-a_\varepsilon(0)u^\varepsilon_x(0,t)]+\int_0^1\lambda u^\varepsilon\,dx=\int_0^1 f_1(u^\varepsilon)\,dx.
$$
By boundary conditions, we obtain
$$
\int_0^1 u^\varepsilon_t\,dx+\int_0^1\lambda u^\varepsilon\,dx=\int_0^1 f_1(u^\varepsilon)\,dx+g_1(v).
$$
Now, we use the well-known argument of \cite{Hale1986,hale_rocha_87} that the large diffusion implies spatial homogenization, that is, we must have $u^\varepsilon(x,t)$ converges as $\varepsilon\to 0$ to a function $u(t)$ independent of the spatial variable. Thus, taking $\varepsilon\to 0$ we obtain
$$
\dot{u}+\lambda u=f_1(u)+g_1(v).
$$
Adding the ODE of \eqref{perturbed_problem}, we obtain in the limiting process the following planar ODE
\begin{equation}\label{unperturbed_equation}
\begin{cases}
\dot{u}=-\lambda u+f_1(u)+g_1(v),\\
\dot{v}=-\beta v+f_2(u)+g_2(v).
\end{cases}
\end{equation}

The aim of this paper is perform a Poincaré compactification of \eqref{unperturbed_equation}, and prove the $C^1$ closeness of the semiflows of \eqref{unperturbed_equation} and \eqref{perturbed_problem} in the Poincaré disk. To achieve this, we project the semiflow of \eqref{perturbed_problem} onto a two-dimensional space $Y_\varepsilon$ identified with $\R^2$. We focus on the case of dissipative polynomial sources, allowing us to modify the source to make it globally Lipschitz without altering the dynamics.

In particular, when polynomial sources are chosen to provide Morse-Smale semiflows, the problem becomes structurally stable. However, this choice appears to be very challenging, or perhaps even impossible under our assumption on $d$, as illustrated in the example in Section \ref{Examples}, where the compactification introduces new equilibrium points that may be non-hyperbolic or exhibit a saddle connection, making the new semiflow no longer Morse-Smale.

We divide this paper as follows: in Section \ref{Technical results}, we present the functional setting required to address \eqref{perturbed_problem} and \eqref{unperturbed_equation}, ensuring that all conditions for applying the Invariant Manifold Theorem are satisfied. In Section \ref{Compactification}, we perform a Poincaré compactification of the semiflows and prove the $C^1$ closeness of the compactified vector fields. In Section \ref{Examples}, we provide an example to illustrate the application of our results.

\section{Invariant Manifold Theorem}\label{Technical results}

We start writing \eqref{perturbed_problem} and \eqref{unperturbed_equation} abstractly as an evolution equation in a suitable Banach space. For this, we denote $L^2_c$ as the set of constant functions in the usual space $L^2(0,1)$ and the Sobolev space $H^1(0,1)$. Our phase spaces are the products $H^1(0,1)\times L^2_c$ and $L^2(0,1)\times L^2_c$ which are closed subsets of $L^2(0,1)\times L^2(0,1)$ with product norm. 

We define for $\varepsilon\in (0,\varepsilon_0]$ the operator  $A_\varepsilon:D(A_\varepsilon)\subset L^2(0,1)\times L^2_c\to L^2(0,1)\times L^2_c$ by
\begin{equation}\label{OP_A}
\begin{cases}
D(A_\varepsilon)=\{(u,v)\in H^1(0,1)\times L^2_c: (a_\varepsilon u_x)_x\in L^2(0,1),\,  u_x(1)=0,\, a_\varepsilon(0)u_x(0)=0\},\\
A_\varepsilon(u,v)=(-(a_\varepsilon u_x)_x+\lambda u ,\beta v)
\end{cases}
\end{equation}
and, for $\varepsilon=0$, we define the operator $A_0:L^2_c\times L_c^2\to L^2_c\times L_c^2$ by $A_0(u,v)=(\lambda u,\beta v)$. 

If we define $F:H^1(0,1)\times L^2_c\to L^2(0,1)\times L^2_c$ by $F(u,v)=(f_1(u)+g_1(v),f_2(u(0))+g_2(v))$, where we have used the same notation for $f_1,f_2,g_1,g_2$ and its  Nemistky functionals, then \eqref{perturbed_problem} and \eqref{unperturbed_equation} can be written as
\begin{equation}\label{abstract_eq}
\begin{cases}
U_t+A_\varepsilon U=F(U),\\
U(0)=U_0,\,\,\,\varepsilon\in [0,\varepsilon_0],
\end{cases}
\end{equation} 
where we are denoting $U=(u,v)$ a point of $H^1(0,1)\times L^2_c$.

Notice that when $\varepsilon=0$, $F$ acts in $L^2_c\times L_c^2$ into itself.

We have $A_\varepsilon$ a positive self-adjoint operator with compact resolvent, $\varepsilon\in (0,\varepsilon_0]$. Hence, the spectrum of $A_\varepsilon$ can be written in the following form
$$
0<\lambda_1^\varepsilon<\lambda_2^\varepsilon<\lambda_3^\varepsilon<\lambda_4^\varepsilon<...
$$ 
Denote $\varphi_1^\varepsilon, \varphi_2^\varepsilon,\varphi_3^\varepsilon,...$ the corresponding eigenfunctions.
 
Next, we summarize some spectral properties of $A_\varepsilon$ whose proof can be found in \cite{Pires_2023_4}.

\begin{theorem}\label{Spectral_behavior}
Let $A_\varepsilon$ be the operator defined in \eqref{OP_A}. Then it is valid for the following properties.
\begin{itemize}
\item[(i)] $\lambda_1^\varepsilon=\beta$ and $\varphi^\varepsilon_1=(0,1)$.
\item[(ii)] $\lambda_2^\varepsilon\to\lambda$ as $\varepsilon\to 0$ and $\varphi^\varepsilon_2\to (1,0)$ as $\varepsilon\to 0$.
\item[(iii)]  $\lambda_k^\varepsilon\to\infty$ as $\varepsilon\to 0$ for $k\geq 3$.
\end{itemize}
\end{theorem}

Now, we introduce some notations to state the Invariant Manifold Theorem. We denote $Y_\varepsilon=span[(0,1),\varphi_2^\varepsilon]$ the space generated by the first two eigenfunctions of $A_\varepsilon$ and we denote $L_\varepsilon$ the spectral projection of $L^2(0,1)\times L^2_c$ onto $Y^\varepsilon$. For each $\varepsilon\in (0,\varepsilon_0]$, we decompose $H^1(0,1)\times L_c^2 =Y_\varepsilon\oplus Y^\perp_\varepsilon$, and we denote $A_\varepsilon^+=A_\varepsilon|_{Y_\varepsilon}$ and  $A_\varepsilon^-=A_\varepsilon|_{Y_\varepsilon^\perp}$. These decompositions allow us to rewrite \eqref{abstract_eq} as the following coupled equation
\begin{equation}\label{coupled_equation}
\begin{cases}
V_t^\varepsilon+A_\varepsilon^+ V^\varepsilon = L_\varepsilon F(V^\varepsilon+Z^\varepsilon):=H_\varepsilon(V^\varepsilon,Z^\varepsilon),\\
Z_t^\varepsilon+A_\varepsilon^- Z^\varepsilon = (I-L_\varepsilon) F(V^\varepsilon+Z^\varepsilon):=G_\varepsilon(V^\varepsilon,Z^\varepsilon).
\end{cases}
\end{equation}

The Invariant Manifold Theorem whose proof can be found in \citep{Lpires1,Anibal1990} and \cite{Santamaria2014} states as follows.
\begin{theorem}\label{invariant_manifold_theorem} For $\varepsilon_0>0$ sufficiently small, there is an invariant manifold $\mathcal{M}_\varepsilon$ for \eqref{abstract_eq} given by
$$
\mathcal{M}_\varepsilon=\{U^\varepsilon\in H^1(0,1)\times L^2_c\,;\, U^\varepsilon = L_\varepsilon U^\varepsilon+s_{*}^\varepsilon(L_\varepsilon U^\varepsilon)\},\quad \varepsilon\in(0,\varepsilon_0],
$$ 
where $s_\ast^\varepsilon: Y_\varepsilon\to Y_\perp^\varepsilon$ is a continuously differentiable map satisfying
\begin{equation*}\label{estimate_invariant_manifold}
\sup_{\tilde{V}^\varepsilon\in Y_\varepsilon}\|s_\ast^\varepsilon(\tilde{V}^\varepsilon)\|_{H^1(0,1)\times L^2_c}+
\sup_{\tilde{V}^\varepsilon\in Y_\varepsilon}\|\dot{s}_\ast^\varepsilon(\tilde{V}^\varepsilon)\|_{\mathcal{L}(Y_\varepsilon,H^1(0,1)\times L^2_c)} \to 0,\quad\hbox{as } \varepsilon\to 0.
\end{equation*}
The invariant manifold $\mathcal{M}_\varepsilon$ is exponentially attracting and the flow of $U_0^\varepsilon\in\mathcal{M}_\varepsilon$ is given by
\begin{equation*}\label{decTTT}
\tilde{T}_\varepsilon(t)U_0^\varepsilon=V^\varepsilon(t)+s_\ast^\varepsilon(V^\varepsilon(t)), \quad t\geq 0,
\end{equation*}
where $V^\varepsilon(t)$ satisfies
\begin{equation}\label{ode_Y}
\begin{cases}
{V_t^\varepsilon}+A_\varepsilon^+V^\varepsilon=H_\varepsilon(V^\varepsilon,s_\ast^\varepsilon(V^\varepsilon)), \quad t\geq 0,\\ V^\varepsilon(0)=L_\varepsilon U_0^\varepsilon\in Y_\varepsilon.
\end{cases}
\end{equation}
\end{theorem}

Notice that the equation \eqref{ode_Y} is an ODE in $Y_\varepsilon$ that describes the behavior of the PDE-ODE \eqref{perturbed_problem}. Moreover, $Y^\varepsilon$ is a two-dimensional space generated by eigenfunctions $(1,0)$ and $\varphi^\varepsilon_2$. 

In what follows, we identify $Y^\varepsilon$ with $\R^2$ and write \eqref{ode_Y} as a planar ODE.

In $L^2(0,1)\times L_c^2$ we consider the following inner product 
\begin{equation}\label{inner_product}
\pin{(u,v),(\varphi,\psi)}=\int_0^1 u\varphi\,dx+\psi v.
\end{equation}
Recall that we are denoting $Y_\varepsilon=span[(0,1),\varphi_2^\varepsilon]$. Since $\varphi_2^\varepsilon\in H^1(0,1)\times L_c^2$ and is orthogonal to $(0,1)$, we have $\varphi_2^\varepsilon=(\varphi^\varepsilon,0)$, for some $\varphi^\varepsilon\in H^1(0,1)$. Hence, we can write $Y_\varepsilon=span[(0,1),(\varphi^\varepsilon,0)]$, where $\varphi^\varepsilon\to 1$ as $\varepsilon\to 0$ according with Theorem \ref{Spectral_behavior}.

The decomposition $H^1(0,1)\times L_c^2=Y_\varepsilon\oplus Y^\perp_\varepsilon$, allows us write a solution $(w^\varepsilon(x,t),v(t)) \in  H^1(0,1)\times L_c^2$ of the PDE-ODE \eqref{perturbed_problem} as,
\begin{equation}\label{tec111}
(w^\varepsilon(x,t),v(t))=w_1^\varepsilon(t)(0,1)+u^\varepsilon(t)(\varphi^\varepsilon(x),0)+(w_\perp^\varepsilon(x,t),0),
\end{equation}
where $(w_\perp^\varepsilon(x,t),0)\in Y_\varepsilon^\perp$.

Taking the inner product \eqref{inner_product} of \eqref{tec111} with $(0,1)$, we obtain
$$
\pin{(w^\varepsilon,v),(0,1)}=\pin{w_1^\varepsilon(0,1),(0,1)}+\pin{u^\varepsilon(\varphi^\varepsilon,0),(0,1)}+\pin{(w_\perp^\varepsilon,0),(0,1)}
$$
which implies $w^\varepsilon_1=v$.

Besides that, if we take the inner product \eqref{inner_product} of \eqref{tec111} with $(\varphi^\varepsilon,0)$, we obtain
$$
\pin{(w^\varepsilon,v),(\varphi^\varepsilon,0)}=\pin{w_1(0,1),(\varphi^\varepsilon,0)}+\pin{u^ \varepsilon(\varphi^\varepsilon,0),(\varphi^\varepsilon,0)}+\pin{(w_\perp^\varepsilon,0),(\varphi^\varepsilon,0)}
$$
which implies $\int_0^1 w^\varepsilon \varphi^\varepsilon\,dx=u^\varepsilon$, where we have used that $\|\varphi^\varepsilon\|_{L^2(0,1)}=1$ and $w^\varepsilon_\perp\perp \varphi^\varepsilon$ in $L^2(0,1)$.

It follows from PDE-ODE \eqref{perturbed_problem} that
$$
\begin{cases}
\dot{u}^\varepsilon(t)=\int_0^1 (a_\varepsilon(x)w^\varepsilon_x(x,t))_x\varphi^\varepsilon(x)\,dx -\int_0^1 \lambda w^\varepsilon(x,t)\varphi^\varepsilon(x)\,dx+\int_0^1 f_1(w^\varepsilon(x,t))\varphi^\varepsilon(x)\,dx,\\
\dot{v}(t)=-\beta v(t)+f_2(w^\varepsilon(0,t))+g_2(v(t)).
\end{cases}
$$
The boundary conditions and the fact that $(\varphi^\varepsilon,0)\in D(A_\varepsilon)$ imply that
\begin{align*}
\int_0^1 (a_\varepsilon w_x^\varepsilon)_x\varphi^\varepsilon\,dx &=-a_\varepsilon(0)w^\varepsilon_x(0)\varphi^\varepsilon(0)-\int_0^1 a_\varepsilon w_x^\varepsilon\varphi_x^\varepsilon\,dx\\
& =-a_\varepsilon(0)w^\varepsilon_x(0)\varphi^\varepsilon(0)+a_\varepsilon(0)\varphi^\varepsilon_x(0)w^\varepsilon(0)+\int_0^1 (a_\varepsilon \varphi_x^\varepsilon)_x w^\varepsilon\,dx\\
&=g_1(v)\varphi^\varepsilon(0)+\int_0^1 (a_\varepsilon \varphi_x^\varepsilon)_x w^\varepsilon\,dx.
\end{align*}
Since $(\varphi^\varepsilon,0)$ is an eigenfunction associated with $\lambda_2^\varepsilon$ and $w^\varepsilon=u^\varepsilon\varphi^\varepsilon+w^\varepsilon_\perp$, we have
\begin{equation}\label{planar_ODE}
\begin{cases}
\dot{u}^\varepsilon=-\lambda_2^\varepsilon u^\varepsilon+\int_0^1 f_1(u^\varepsilon\varphi^\varepsilon+w_\perp^\varepsilon)\varphi^\varepsilon\,dx+g_1(v)\varphi^\varepsilon(0),\\
\dot{v}=-\beta v+f_2(u^\varepsilon\varphi^\varepsilon(0)+w_\perp^\varepsilon(0,\cdot))+g_2(v).
\end{cases}
\end{equation}

The equation \eqref{planar_ODE} is the version of the equation \eqref{ode_Y} to $\R^2$.

The next result ensures that the behavior of the PDE-ODE \eqref{perturbed_problem} is dictated by the ODE \eqref{planar_ODE}.
\begin{theorem}\label{conv_1222} Let $(w^\varepsilon,v)$ be a solution of \eqref{perturbed_problem} and $w^\varepsilon_\perp$ be as in \eqref{tec111}. Then it is valid the following convergence
$$
\sup_{t\in (0,1)}\|w^\varepsilon_\perp(x,t)\|_{H^1(0,1)}\to 0\quad\hbox{as}\quad \varepsilon\to 0.
$$
\end{theorem}
\begin{proof}
By \eqref{tec111} we have that $w^\varepsilon=u^\varepsilon\varphi^\varepsilon+w^\varepsilon_\perp$. Since $w^\varepsilon$ is a solution of \eqref{perturbed_problem} and $u^\varepsilon$ satisfies \eqref{planar_ODE}, we have
\begin{align*}
(w^\varepsilon_\perp)_t&=w^\varepsilon_t-\dot{u}^\varepsilon\varphi^\varepsilon\\
&=(a_\varepsilon w^\varepsilon_x)_x-\lambda w^\varepsilon+f_1(w^\varepsilon)+\lambda_2^\varepsilon u^\varepsilon\varphi^\varepsilon-\int_0^1 f_1(u^\varepsilon\varphi^\varepsilon+w_\perp^\varepsilon)\varphi^\varepsilon\,dy\varphi^\varepsilon-g_1(v)\varphi^\varepsilon(0)\varphi^\varepsilon\\
&=(a_\varepsilon (w_\perp^\varepsilon)_x)_x-\lambda w_\perp^\varepsilon+f_1(u^\varepsilon\varphi^\varepsilon+w^\varepsilon_\perp)-\int_0^1 f_1(u^\varepsilon\varphi^\varepsilon+w_\perp^\varepsilon)\varphi^\varepsilon\,dy\varphi^\varepsilon-g_1(v)\varphi^\varepsilon(0)\varphi^\varepsilon,
\end{align*}
whew we have used $(a_\varepsilon\varphi^\varepsilon_x)_x-\lambda \varphi^\varepsilon+\lambda^\varepsilon_2\varphi^\varepsilon=0$.

Hence, $w^\varepsilon_\perp$ is a solution of the following parabolic equation
\begin{equation}\label{parabolic_eq}
\begin{cases}
(w^\varepsilon_\perp)_t-(a_\varepsilon (w^\varepsilon_\perp)_x)_x+\lambda w^\varepsilon_\perp =R_\varepsilon(u^\varepsilon,v,w_\perp^\varepsilon),\quad &\hbox{in }(0,1)\times (0,\infty),\\
(w^\varepsilon_\perp)_x(1)=0, &\hbox{in }(0,\infty),\\
-a_\varepsilon(0)(w^\varepsilon_\perp)_x(0)=g_1(v),&\hbox{in }(0,\infty),\\
w^\varepsilon_\perp(0)=w_0, &\hbox{in }(0,1),
\end{cases}
\end{equation}
where
$$
R_\varepsilon(u^\varepsilon,v,w_\perp^\varepsilon)=f_1(u^\varepsilon\varphi^\varepsilon+w^\varepsilon_\perp)-\int_0^1 f_1(u^\varepsilon\varphi^\varepsilon+w_\perp^\varepsilon)\varphi^\varepsilon\,dy\varphi^\varepsilon-g_1(v)\varphi^\varepsilon(0)\varphi^\varepsilon.
$$

We define the operator $B_\varepsilon:D(B_\varepsilon)\subset L^2(0,1)\to L^2(0,1)$ by $D(B_\varepsilon)=\{u\in H^1(0,1): (a_\varepsilon u_x)_x\in L^2(0,1),\,u_x(1)=0,\,a_\varepsilon(0)u_x(0)=0\}$ and $B_\varepsilon u= -(a_\varepsilon u_x)_x+\lambda u$. Since $w_\perp^\varepsilon$ is a solution of \eqref{parabolic_eq}, it must satisfies the following variation of constants formula,
$$
w_\perp^\varepsilon(t)=e^{-B_\varepsilon t}w_0+\int_0^t e^{-B_\varepsilon(t-s)}R_\varepsilon(u^\varepsilon,v,w_\perp^\varepsilon)\,ds.
$$

As in \cite[eq. (2.7)]{Hale1986},  we can take constants $C_1>0$ and $\mu>0$ independent of $\varepsilon$, such that the linear semiflow $\{e^{-B_\varepsilon t}\}_{t\geq 0}$ restrict to $Y_\perp^\varepsilon$ satisfy the properties
\begin{equation}\label{exp_estimate}
\begin{cases}
\|e^{-B_\varepsilon t}w_0\|_{H^1(0,1)}\leq C_1t^{-\frac{1}{2}}e^{-\mu\tau(\varepsilon)t} \|w_0\|_{L^2(0,1)},\\
\|e^{-B_\varepsilon t}w_0\|_{H^1(0,1)}\leq C_1 e^{-\mu\tau(\varepsilon)t}\|w_0\|_{H^1(0,1)},
\end{cases}
\end{equation}
where $\tau(\varepsilon)$ is given by \eqref{large_diffusion}.

Moreover, we can found constants $K_1,K_2$ independents of $\varepsilon$ such that
\begin{align*}
\|R_\varepsilon(u^\varepsilon,v,w_\perp^\varepsilon)\|_{L^2(0,1)}&\leq |\int_0^1f_1(u^\varepsilon\varphi^\varepsilon(x)+w^\varepsilon_\perp(x))- f_1(u^\varepsilon\varphi^\varepsilon(y)+w_\perp^\varepsilon(y))\varphi^\varepsilon(y)|\,dy+|g_1(v)\varphi^\varepsilon(0)|\\
&\leq \int_0^1|f_1(u^\varepsilon\varphi^\varepsilon(x)+w^\varepsilon_\perp(x))||1-\varphi^\varepsilon(y)|\,dy\\
&+\int_0^1|f_1(u^\varepsilon\varphi^\varepsilon(x)+w_\perp^\varepsilon(x))-f_1(u^\varepsilon\varphi^\varepsilon(y)+w_\perp^\varepsilon(y))||\varphi^\varepsilon(y)|\,dy+|g_1(v)\varphi^\varepsilon(0)|\\
&\leq o(1)+K_2+K_1\|w^\varepsilon_\perp\|_{H^1(0,1)},
\end{align*}
where $o(1)$ denotes a function that goes to zero as $\varepsilon\to 0$ and, we have used Theorem \ref{Spectral_behavior} and the continuity of $f_1$ and $g_1$ to guarantee the boundness constants.

Hence, by \eqref{exp_estimate}, we have
\begin{align*}
\|w_\perp^\varepsilon(t)\|_{H^1(0,1)}&\leq C_1e^{-\mu\tau(\varepsilon)t}\|w_0\|_{H^1(0,1)}+\int_0^t e^{-\mu\tau(\varepsilon)(t-s)}(t-s)^{-\frac{1}{2}}\|R_\varepsilon(u^\varepsilon,v,w_\perp^\varepsilon)\|_{H^1(0,1)}\,ds\\
&\leq C_1e^{-\mu\tau(\varepsilon)t}\|w_0\|_{H^1(0,1)}+(o(1)+K_2)\int_0^t(t-s)^\frac{1}{2}e^{-\mu\tau(\varepsilon)(t-s)}\,ds\\
& +K_1\int_0^t(t-s)^\frac{1}{2}e^{-\mu\tau(\varepsilon)(t-s)}\|w^\varepsilon_\perp(s)\|_{H^1(0,1)}\,ds.
\end{align*}
The result follows applying the Gronwall inequality.
%{\it Steps 1. Linear Convergence.} The exponential is given by 
%$$
%e^{-B_\varepsilon t}w_0=\sum_{i=3}^\infty e^{-\lambda^\varepsilon_i t}\pin{w_0,\psi^\varepsilon_i}\psi^\varepsilon_i,\quad t>0,
%$$
%where we denote $\varphi_i^\varepsilon=(\psi^\varepsilon_i,0)$.
%
%But, $\lambda_3^\varepsilon\leq \lambda_i^\varepsilon$ for $i=3,4,...$, which implies $e^{-\lambda_i^\varepsilon t}\leq e^{-\lambda_3^\varepsilon t} $, for $t>0$ and then,
%$$
%\|e^{-B_\varepsilon t}w_0\|_{H^1(0,1)}=\Big(e^{-2\lambda_3^\varepsilon t}\sum_{i=2}^\infty \pin{w_0,\varphi_i^\varepsilon}\lambda_i^\varepsilon\Big)^\frac{1}{2}\leq M e^{-\lambda_3^\varepsilon t}\|w_0\|_{X_\varepsilon^\frac{1}{2}},\quad t>0.
%$$
%The function $h(\eta)=e^{-2\eta t}\eta$ attains its maximum at $\eta=1/2t$, $t>0$. Then,
%$$
%\|e^{-B_\varepsilon  t}w_0\|_{X_\varepsilon^\frac{1}{2}}\leq \begin{cases} e^{-\lambda_3^\varepsilon t}(\lambda_3^\varepsilon)^\frac{1}{2}\|z\|_{L^2}, \,\,\,1/2t<\lambda_2^\varepsilon, \\ e^{-\lambda_2^\varepsilon t}2^{-\frac{1}{2}}t^{-\frac{1}{2}} \|z\|_{L^2},\,\,\,1/2t>\lambda^\varepsilon_2.\end{cases}
%$$

\end{proof}

%Using that $\theta^\varepsilon(u^\varepsilon\varphi^\varepsilon)=w^\varepsilon_\perp$, we have
%\begin{equation}
%\begin{cases}
%\dot{u}^\varepsilon=-\lambda_2^\varepsilon u^\varepsilon+\int_0^1 f_1(u^\varepsilon\varphi^\varepsilon+\theta^\varepsilon(u^\varepsilon\varphi^\varepsilon))\varphi^\varepsilon\,dx+g_1(v)\varphi^\varepsilon(0),\\
%\dot{v}=-\beta v+f_2(u^\varepsilon\varphi^\varepsilon(0)+\theta^\varepsilon(u^\varepsilon\varphi^\varepsilon(0)))+g_2(v).
%\end{cases}
%\end{equation}

For each $\varepsilon\in[0,\varepsilon_0]$, we define the planar vector field $X_\varepsilon(u,v)=(P_\varepsilon(u,v),Q_\varepsilon(u,v))$, where
\begin{equation}\label{polynomial_vector_field_1}
\begin{cases}
P_0(u,v)= -\lambda u+f_1(u)+g_1(v),  \\
Q_0(u,v)=-\beta v+f_2(u)+g_2(v),
\end{cases}
\end{equation}
and
\begin{equation}\label{polynomial_vector_field_2}
\begin{cases}
P_\varepsilon(u,v)= -\lambda_2^\varepsilon u+\int_0^1 f_1(u\varphi^\varepsilon+\theta^\varepsilon(u\varphi^\varepsilon))\varphi^\varepsilon\,dx+g_1(v)\varphi^\varepsilon(0),  \\
Q_\varepsilon(u,v)=-\beta v+f_2(u\varphi^\varepsilon(0)+\theta^\varepsilon(u\varphi^\varepsilon(0)))+g_2(v).
\end{cases}
\end{equation}

%Since we do not have a compact set containing all dynamics of the problem, taking the $C^1$ norm $\|X_\varepsilon-X_0\|_1$ in compact sets does not produce suitable comparison due to the possibility of the dynamics at infinite. Taking the $C^1$ norm $\|X_\varepsilon-X_0\|_1$ in all $\R^2$ is not a good strategy due to the unboundedness of the polynomial sources. In the next section, we will find a way to overcome this issue.

%{\color{red}
%We denote $B(1)$ the unit open ball centered in the origin of $\R^2$. Recall that the $C^1$-norm of a planar $C^1$ vector field $X:\R^2\to \R^2$ is defined by
%$$
%\|X\|_1=\sup_{U\in B(1)}\{\|X(U)\|_{\R^2},\|X^\prime(U)\|_{\mathcal{L}(\R^2)}\}.
%$$ 
%
%\begin{theorem}
%Let $X_\varepsilon$ and $X_0$ be the polynomial vector fields defined in \eqref{polynomial_vector_field_1} and \eqref{polynomial_vector_field_2}, respectively. Then it is valid for the following convergence
%$$
%\|X_\varepsilon-X_0\|_1\to 0\quad\hbox{as}\quad\varepsilon\to 0.
%$$
%\end{theorem}
%\begin{proof}
%
%\end{proof}
%}

\section{Poincaré compactification}\label{Compactification}

In this section, we perform a Poincaré compactification for semiflows of \eqref{perturbed_problem} and \eqref{unperturbed_equation} and we prove the $C^1$ closeness of the respective compactified vector fields. Our main reference for Poincaré compactification is \cite{Jaume_book} we also refer to the works \cite{Jaume1990,Jaum1995}.

Recall that we are assuming the nonlinearities $f_1$, $f_2$, $g_1$, and $g_2$ are polynomials such that the maximum degree $d$ among $f_1$, $f_2$, $g_1$, and $g_2$ satisfies $d - 2 \geq 0$, and $d - 1$ is greater than or equal to the degrees of $g_1$ and $f_1$. Moreover, $f_1$, $f_2$, $g_1$, and $g_2$ satisfy the dissipative condition \eqref{f_assumption}. Under this assumption, it is well known from \cite{Carvalho_1995} that \eqref{unperturbed_equation} and \eqref{planar_ODE} possess global attractors, denoted by $\mathcal{A}$ and $\mathcal{A}_\varepsilon$, respectively.

Following the approach in \cite{J.M.Arrieta2000}, we define $K = \sup\{\|u\|_{L^\infty} : u \in \cup_{\varepsilon \geq 0} \mathcal{A}_\varepsilon \}$, which is finite because $\cup_{\varepsilon \geq 0} \mathcal{A}_\varepsilon$ is uniformly bounded. Next, we truncate $f_1$, $f_2$, $g_1$, and $g_2$ outside $(-2K, 2K)$, obtaining new polynomial nonlinearities $\tilde{f}_1$, $\tilde{f}_2$, $\tilde{g}_1$, and $\tilde{g}_2$. These truncated nonlinearities retain their degrees while becoming globally Lipschitz in $\R^2$. In what follows, we continue to use the original notations $f_1$, $f_2$, $g_1$, and $g_2$ for simplicity.

Let $M\subset\R^3$ be a compact differentiable two-dimensional manifold. Consider an open cover $M=V_1\cup...\cup V_k$ with $V_i$ contained in a local chart $(\phi_i,U_i)$ such that $\phi_i(U_i)=B(2)$ and $\phi_i(V_i)=B(1)$, where $B(1)$ and $B(2)$ are open balls centered in the origin of $\R^2$ of radius $1$ and $2$, respectively. The $C^1$-norm of a vector field $X$ in $M$ is defined by
$$
\|X\|_1=\max_{1\leq i\leq k}\sup_{(x,y)\in B(1)}\{\|X(\phi^{-1}_i(x,y))\|_{\R^2},\|X^\prime(\phi^{-1}_i(x,y))\|_{\mathcal{L}(\R^2)} \}.
$$

Let $X_\varepsilon(u^\varepsilon,v)=(P_\varepsilon(u^\varepsilon,v),Q_\varepsilon(u^\varepsilon,v))$, $\varepsilon\in [0,\varepsilon_0]$, be the polynomial vector field defined in \eqref{polynomial_vector_field_1} and \eqref{polynomial_vector_field_2} and denote $d$ the maximum degree of $P_\varepsilon$ and $Q_\varepsilon$.

We embedded $\R^2$ into $\R^3$ as the plane $\{(u,v,1):(u,v)\in\R^2\}$ tangent to $\mathbb{S}^2$ at point $(0,0,1)$. We decompose $\mathbb{S}^2=\mathbb{H}^+\cup\mathbb{S}^1\cup \mathbb{H}^-$, whee $\mathbb{H}^+$ is the northern hemisphere, $\mathbb{H}^-$ is the southern hemisphere and $\mathbb{S}^1$ is the equator. 

The compactification of the vector field $X_\varepsilon:\R^2\to \R^2$ to the sphere $\mathbb{S}^2$ is given by central projections $f^+,f^-:\R^2\to \mathbb{S}^2$, $f^-=-f^+$, where
$$
f^+(u,v)=\Big(\frac{u}{\Delta(u,v)},\frac{v}{\Delta(u,v)},\frac{1}{\Delta(u,v)}\Big), \quad \Delta(u,v)=\sqrt{u^2+v^2+1}.
$$
In $\mathbb{S}^2$ we consider the usual six local charts $\mathbb{S}^2=U_1\cup U_2\cup...\cup U_6$ given by 
$$
U_1=\{(y_1,y_2,y_3)\in \mathbb{S}^2: y_1>0\}, U_2=\{(y_1,y_2,y_3)\in \mathbb{S}^2: y_2>0\}, U_3=\{(y_1,y_2,y_3)\in \mathbb{S}^2: y_3>0\}
$$
$$
U_4=\{(y_1,y_2,y_3)\in \mathbb{S}^2: y_1<0\}, U_5=\{(y_1,y_2,y_3)\in \mathbb{S}^2: y_2<0\}, U_6=\{(y_1,y_2,y_3)\in \mathbb{S}^2: y_3<0\}.
$$
and local maps are given by
$$
\phi_1(u,v,w)=\Big(\frac{v}{u}, \frac{w}{u} \Big),\phi_2(u,v,w)=\Big(\frac{u}{v}, \frac{w}{v} \Big),\phi_3(u,v,w)=\Big(\frac{u}{w}, \frac{v}{w} \Big)
$$
$$
\phi_4(u,v,w)=\Big(-\frac{v}{u}, -\frac{w}{u} \Big),\phi_5(u,v,w)=\Big(-\frac{u}{v}, -\frac{w}{v} \Big),\phi_6(u,v,w)=\Big(-\frac{u}{w}, -\frac{v}{w} \Big).
$$

In order to study  $X_\varepsilon$ in $\R^2$ including its behavior near infinity it is enough working on the Poincaré disk $\mathbb{D}^2=\mathbb{H}^+\cup \mathbb{S}^1$. The compactification $\bar{X}_\varepsilon:\mathbb{D}^2\to\mathbb{D}^2$ of the vector field $X_\varepsilon:\R^2\to \R^2$  is given by
\begin{equation}\label{CompXX}
\bar{X}_\varepsilon(Y)=\rho(Y)Df^+(X)X_\varepsilon(X),\quad Y=F^+(X),
\end{equation}
where $\rho(Y)$ is a bounded correction therm.

According with \cite[Chapter 5]{Jaume_book}, the expression \eqref{CompXX} in the charts $(\phi_i,U_i)$, $i=1,...,6$ is given as follows.

In the chart $(\phi_1,U_1)$ make the change of variable $(u,v)=(1/x,x/y)$ to obtain 
\begin{equation}\label{edo1}
\begin{cases}
\dot{x}=y^d\Big[ -x P_\varepsilon\Big(\frac{1}{y},\frac{x}{y}\Big)+Q_\varepsilon\Big(\frac{1}{y},\frac{x}{y}\Big)\Big],    \\
\dot{y}=-y^{d+1} P_\varepsilon\Big(\frac{1}{y},\frac{x}{y}\Big).
\end{cases}
\end{equation}

In the chart $(\phi_2,U_2)$ make the change of variable $(u,v)=(x/y,1/x)$ to obtain
\begin{equation}\label{edo2}
\begin{cases}
\dot{x}=y^d\Big[  P_\varepsilon\Big(\frac{x}{y},\frac{1}{y}\Big)-xQ_\varepsilon\Big(\frac{x}{y},\frac{1}{y}\Big)\Big],    \\
\dot{y}=-y^{d+1} Q_\varepsilon\Big(\frac{x}{y},\frac{1}{y}\Big).
\end{cases}
\end{equation}

In the chart $(\phi_3,U_3)$  make the change of variable $(u,v)=(x,y)$ to obtain
\begin{equation}\label{edo3}
\begin{cases}
\dot{x}=P_\varepsilon(x,y),  \\
\dot{y}=Q_\varepsilon(x,y).
\end{cases}
\end{equation}

Moreover, the expression for $\bar{X}_\varepsilon$ in the charts $(\phi_4,U_4)$, $(\phi_5,U_5)$ and $(\phi_6,U_6)$ is the same as for $(\phi_1,U_1)$, $(\phi_2,U_2)$ and $(\phi_3,U_3)$ multiplied by $(-1)^d$, respectively. Therefore, all calculations can be done in the first three charts.

The phase portrait on the Poincaré disk $\mathbb{D}^2$ is constructed as follows: The ODE \eqref{edo3} provides the equilibria inside $\mathbb{D}^2$, which coincide with the equilibria of the ODE in $\mathbb{R}^2$. The equations \eqref{edo1} and \eqref{edo2}, by taking $y=0$, yield the equilibria at infinity, which correspond to the equilibria on the equator $\mathbb{S}^1$, the boundary of $\mathbb{D}^2$.

\begin{theorem}\label{theoC1}
Let $\bar{X}_\varepsilon$ and $\bar{X}_0$ be the respective compactification to the Poncaré disk $\mathbb{D}^2$ of the polynomial vector fields $X_\varepsilon$ and $X_0$ defined in \eqref{polynomial_vector_field_1} and \eqref{polynomial_vector_field_2}. Then it is valid the following convergence
\begin{equation}\label{convC1}
\|\bar{X}_\varepsilon-\bar{X}_0\|_1\to 0\quad\hbox{as}\quad\varepsilon\to 0.
\end{equation}
\end{theorem}
\begin{proof}
Recall that
$$
\begin{cases}
P_\varepsilon(u,v)= -\lambda_2^\varepsilon u+\int_0^1 f_1(u\varphi^\varepsilon(r)+w_\perp^\varepsilon(r,\cdot))\varphi^\varepsilon(r)\,dr+g_1(v)\varphi^\varepsilon(0),\\
Q_\varepsilon(u,v)=-\beta v+f_2(u\varphi^\varepsilon(0)+w_\perp^\varepsilon(0,\cdot))+g_2(v),
\end{cases}
$$
and 
$$
\begin{cases}
P_0(u,v)=-\lambda u+f_1(u)+g_1(v),\\
Q_0(u,v)=-\beta v+f_2(u)+g_2(v).
\end{cases}
$$

\noindent {\it Step 1. $C^0$-convergence.}

In the chart $(\phi_1,U_1)$, $P_\varepsilon$ and $P_0$, read as 
\begin{multline*}
\bar{P}_\varepsilon(x,y)=y^d\Big[-x\Big[-\lambda_2^\varepsilon\frac{1}{y}+ \int_0^1 f_1\Big(\frac{1}{y}\varphi^\varepsilon(r)+w_\perp^\varepsilon(r,\cdot)\Big)\varphi^\varepsilon(r)\,dr+g_1\Big(\frac{x}{y}\Big)\varphi^\varepsilon(0)\Big]\\
+\Big(-\beta\frac{x}{y}+f_2\Big(\frac{1}{y}\varphi^\varepsilon(0)+w_\perp^\varepsilon(0,\cdot)\Big) +g_2\Big(\frac{x}{y}\Big) \Big)    \Big]
\end{multline*}
and
$$
\bar{P}_0(x,y)=y^d\Big[-x\Big[-\lambda\frac{1}{y}+f_1\Big(\frac{1}{y}\Big)+g_1\Big(\frac{x}{y}\Big)\Big]
+\Big(-\beta\frac{x}{y}+f_2\Big(\frac{1}{y}\Big) +g_2\Big(\frac{x}{y}\Big) \Big)    \Big]. 
$$
Thus,
\begin{align*}
|\bar{P}_\varepsilon(x,y)-\bar{P}_0(x,y)|&\leq |\lambda-\lambda_2^\varepsilon||xy^{d-1}|+\int_0^1  |xy^d||f_1\Big(\frac{1}{y}\varphi^\varepsilon(r)+w_\perp^\varepsilon(r,\cdot)\Big)\varphi^\varepsilon(r)-f_1\Big(\frac{1}{y}\Big)| \,dr\\
&+|xy^dg_1\Big(\frac{x}{y}\Big)||\varphi^\varepsilon(0)-1|+|y^d||f_2\Big(\frac{1}{y}\varphi^\varepsilon(0)+w_\perp^\varepsilon(0,\cdot)\Big)-f_2\Big(\frac{1}{y}\Big)|.
\end{align*}
But,
\begin{align*}
\int_0^1 & |xy^d||f_1\Big(\frac{1}{y}\varphi^\varepsilon(r)+w_\perp^\varepsilon(r,\cdot)\Big)\varphi^\varepsilon(r)-f_1\Big(\frac{1}{y}\Big)|\,dr\leq \int_0^1|xy^d||f_1\Big(\frac{1}{y}\Big)|1-\varphi^\varepsilon(r)|\,dr\\
&+\int_0^1 K_1 |xy^d||\frac{1}{y}\varphi^\varepsilon(r)+w_\perp^\varepsilon(r,\cdot)-\frac{1}{y}|\,dr\\
&\leq \int_0^1|xy^d||f_1\Big(\frac{1}{y}\Big)|1-\varphi^\varepsilon(r)|\,dr+ \int_0^1 K_1 |xy^{d-1}||\varphi^\varepsilon(r)-1|+K_1 |xy^{d}||w_\perp^\varepsilon(r,\cdot)|\,dr,
\end{align*}
where $K_1$ depends on the Lipschitz constant of $f_1$ and the boundness of $\varphi^\varepsilon$ on $B(1)$.

In a similar way,
$$
|y^d||f_2\Big(\frac{1}{y}\varphi^\varepsilon(0)+w_\perp^\varepsilon(0,\cdot)\Big)-f_2\Big(\frac{1}{y}\Big)|\leq K_2|y^{d-1}||\varphi^\varepsilon(0)-1|+K_2 |y^{d}||w_\perp^\varepsilon(0,\cdot)|,
$$
where $K_2$ denotes the Lipschitz constant of $f_2$ on $B(1)$.

Since $d$ is greater than or equal degree of $g_1$ and $f_1$, we have that $|xy^{d}f_1(1/y)|$ and $|xy^{d}g_1(x/y)|$ are bounded in $B(1)$. Moreover, $d-1\geq 0$ and, by Theorems \ref{Spectral_behavior} and \ref{conv_1222}, $\lambda_2^\varepsilon\to \lambda$, $\varphi^\varepsilon\to 1$ and $w^\varepsilon_\perp\to 0$ as $\varepsilon\to 0$. Putting all estimates together, we obtain 
$$
\sup_{(x,y)\in B(1)}|\bar{P}_\varepsilon(x,y)-\bar{P}_0(x,y)|\to 0\quad\hbox{as}\quad\varepsilon\to 0.
$$

Moreover, in the chart $(\phi_1,U_1)$, $Q_\varepsilon$ and $Q_0$, read as 
$$
\bar{Q}_\varepsilon(x,y)=-y^{d+1}\Big[ -\lambda_2^\varepsilon \frac{1}{y}+\int_0^1 f_1\Big(\frac{1}{y}\varphi^\varepsilon(r)+w_\perp^\varepsilon(r,\cdot)\Big)\varphi^\varepsilon(r)\,dr+g_1\Big(\frac{x}{y}\Big)\varphi^\varepsilon(0)   \Big]
$$
and
$$
\bar{Q}_0(x,y)=-y^{d+1}\Big[ -\lambda \frac{1}{y}+f_1\Big(\frac{1}{y}\Big)+g_1\Big(\frac{x}{y}\Big)\Big].
$$
Thus,
\begin{align*}
|\bar{Q}_\varepsilon(x,y)-\bar{Q}_0(x,y)|&\leq |\lambda-\lambda_2^\varepsilon||y^{d}|+\int_0^1  |y^{d+1}||f_1\Big(\frac{1}{y}\varphi^\varepsilon(r)+w_\perp^\varepsilon(r,\cdot)\Big)\varphi^\varepsilon(r)-f_1\Big(\frac{1}{y}\Big)| \,dr\\
&+|y^{d+1}g_1\Big(\frac{x}{y}\Big)||\varphi^\varepsilon(0)-1|.
\end{align*}
But
\begin{align*}
\int_0^1 & |y^{d+1}||f_1\Big(\frac{1}{y}\varphi^\varepsilon(r)+w_\perp^\varepsilon(r,\cdot)\Big)\varphi^\varepsilon(r)-f_1\Big(\frac{1}{y}\Big)|\,dr\leq \int_0^1|y^{d+1}||f_1\Big(\frac{1}{y}\Big)|1-\varphi^\varepsilon(r)|\,dr\\
&+\int_0^1 K_1 |y^{d+1}||\frac{1}{y}\varphi^\varepsilon(r)+w_\perp^\varepsilon(r,\cdot)-\frac{1}{y}|\,dr\\
&\leq \int_0^1|y^{d+1}||f_1\Big(\frac{1}{y}\Big)|1-\varphi^\varepsilon(r)|\,dr+\int_0^1 K_1 |y^{d}||\varphi^\varepsilon(r)-1|+K_1 |y^{d+1}||w_\perp^\varepsilon(r,\cdot)|\,dr,
\end{align*}
where $K_1$ depends on the Lipschitz constant of $f_1$ and the boundness of $\varphi^\varepsilon$ on $B(1)$.

Since $d$ is greater than or equal degree of $f_1$ and $g_1$, we have that $|y^{d+1}f_1(1/y)|$ and $|y^{d+1}g_1(x/y)|$ are bounded in $B(1)$. By convergence properties $\lambda_2^\varepsilon\to \lambda$, $\varphi^\varepsilon\to 1$ and $w^\varepsilon_\perp\to 0$ as $\varepsilon\to 0$, we obtain
$$
\sup_{(x,y)\in B(1)}|\bar{Q}_\varepsilon(x,y)-\bar{Q}_0(x,y)|\to 0\quad\hbox{as}\quad\varepsilon\to 0.
$$

In the chart $(\phi_2,U_2)$, $P_\varepsilon$ and $P_0$ reads as
\begin{multline*}
\bar{P}_\varepsilon(x,y)=y^d\Big[-\lambda_2^\varepsilon\frac{x}{y}+ \int_0^1 f_1\Big(\frac{x}{y}\varphi^\varepsilon(r)+w_\perp^\varepsilon(r,\cdot)\Big)\varphi^\varepsilon(r)\,dr+g_1\Big(\frac{1}{y}\Big)\varphi^\varepsilon(0)\\
-x\Big(-\beta\frac{1}{y}+f_2\Big(\frac{x}{y}\varphi^\varepsilon(0)+w_\perp^\varepsilon(0,\cdot)\Big) +g_2\Big(\frac{1}{y}\Big) \Big)    \Big] 
\end{multline*}
and
$$
\bar{P}_0(x,y)=y^d\Big[-\lambda\frac{x}{y}+f_1\Big(\frac{x}{y}\Big)+g_1\Big(\frac{1}{y}\Big)
-x\Big(-\beta\frac{1}{y}+f_2\Big(\frac{x}{y}\Big) +g_2\Big(\frac{1}{y}\Big) \Big)\Big]. 
$$
Thus,
\begin{align*}
|\bar{P}_\varepsilon(x,y)-\bar{P}_0(x,y)|&\leq |\lambda-\lambda_2^\varepsilon||xy^{d-1}|+\int_0^1  |y^d||f_1\Big(\frac{x}{y}\varphi^\varepsilon(r)+w_\perp^\varepsilon(r,\cdot)\Big)\varphi^\varepsilon(r)-f_1\Big(\frac{x}{y}\Big)| \,dr\\
&+|y^dg_1\Big(\frac{1}{y}\Big)||\varphi^\varepsilon(0)-1|+|xy^d||f_2\Big(\frac{x}{y}\varphi^\varepsilon(0)+w_\perp^\varepsilon(0,\cdot)\Big)-f_2\Big(\frac{x}{y}\Big)|
\end{align*}
But
\begin{align*}
\int_0^1 & |y^d||f_1\Big(\frac{x}{y}\varphi^\varepsilon(r)+w_\perp^\varepsilon(r,\cdot)\Big)\varphi^\varepsilon(r)-f_1\Big(\frac{x}{y}\Big)|\,dr\leq \int_0^1|y^{d}||f_1\Big(\frac{x}{y}\Big)|1-\varphi^\varepsilon(r)|\,dr\\
&+\int_0^1 K_1 |y^d||\frac{x}{y}\varphi^\varepsilon(r)+w_\perp^\varepsilon(r,\cdot)-\frac{x}{y}|\,dr\\
&\leq \int_0^1|y^{d}||f_1\Big(\frac{x}{y}\Big)|1-\varphi^\varepsilon(r)|\,dr+\int_0^1 K_1 |xy^{d-1}||\varphi^\varepsilon(r)-1|+K_1 |y^{d}||w_\perp^\varepsilon(r,\cdot)|\,dr
\end{align*}
and
\begin{align*}
|xy^d||f_2\Big(\frac{x}{y}\varphi^\varepsilon(0)+w_\perp^\varepsilon(0,\cdot)\Big)-f_2\Big(\frac{x}{y}\Big)|\leq K_2 |x^2y^{d-1}||\varphi^\varepsilon(0)-1|+K_2 |xy^{d}||w_\perp^\varepsilon(0,\cdot)|,
\end{align*}
where $K_2$ denotes the Lipschitz constant of $f_2$ on $B(1)$.

Since $d$ is greater than or equal degree of $g_1$ and $f_1$, we have that $|y^{d}f_1(x/y)|$ and $|y^{d}g_1(1/y)|$ are bounded in $B(1)$. By convergence properties $\lambda_2^\varepsilon\to \lambda$, $\varphi^\varepsilon\to 1$ and $w^\varepsilon_\perp\to 0$ as $\varepsilon\to 0$, we obtain
$$
\sup_{(x,y)\in B(1)}|\bar{P}_\varepsilon(x,y)-\bar{P}_0(x,y)|\to 0\quad\hbox{as}\quad\varepsilon\to 0.
$$
Moreover, in the chart $(\phi_2,U_2)$, $Q_\varepsilon$ and $Q_0$ reads as
$$
\bar{Q}_\varepsilon(x,y)=-y^{d+1}\Big[ -\beta \frac{1}{y}+f_2\Big(\frac{x}{y} \varphi^\varepsilon(0)+w_\perp^\varepsilon(0,\cdot)\Big)+g_2\Big(\frac{1}{y}\Big)\Big]
$$
and
$$
\bar{Q}_0(x,y)=-y^{d+1}\Big[-\beta \frac{1}{y}+f_2\Big(\frac{x}{y}\Big)+g_2\Big(\frac{1}{y}\Big)\Big].
$$
Thus,
\begin{align*}
|\bar{Q}_\varepsilon(x,y)-\bar{Q}_0(x,y)|&\leq |y^{d+1}||f_2\Big(\frac{x}{y}\varphi^\varepsilon(0)+w_\perp^\varepsilon(0,\cdot)\Big)-f_2\Big(\frac{x}{y}\Big)|\\
&\leq K_2 |xy^{d}||\varphi^\varepsilon(0)-1|+K_2 |y^{d+1}||w_\perp^\varepsilon(0,\cdot)|.
\end{align*}
By convergence properties  $\varphi^\varepsilon\to 1$ and $w^\varepsilon_\perp\to 0$ as $\varepsilon\to 0$, we obtain
$$
\sup_{(x,y)\in B(1)}|\bar{Q}_\varepsilon(x,y)-\bar{Q}_0(x,y)|\to 0\quad\hbox{as}\quad\varepsilon\to 0.
$$

In the chart $(\phi_3,U_3)$, $P_\varepsilon$ and $P_0$ reads as
$$
\bar{P}_\varepsilon(x,y)= -\lambda_2^\varepsilon x+\int_0^1 f_1(x\varphi^\varepsilon(r)+w_\perp^\varepsilon(r,\cdot))\varphi^\varepsilon(r)\,dr+g_1(y)\varphi^\varepsilon(0)
$$
and 
$$
\bar{P}_0(x,y)=-\lambda x+f_1(x)+g_1(y).
$$
Thus,
\begin{align*}
|\bar{P}_\varepsilon(x,y)-\bar{P}_0(x,y)|&\leq \int_0^1 |f_1(x\varphi^\varepsilon(r)+w_\perp^\varepsilon(r,\cdot))\varphi^\varepsilon(r)-f_1(x)|\,dr\\
&\leq \int_0^1 |f_1(x)||1-\varphi^\varepsilon(r)|\,dr+\int_0^1 K_1 |x\varphi^\varepsilon(r)+w_\perp^\varepsilon(r,\cdot)-x|\,dr\\
&\leq \int_0^1 K_1 |x||\varphi^\varepsilon(r)-1|+K_1|w_\perp^\varepsilon(r,\cdot)|\,dr.
\end{align*}
By convergence properties $\varphi^\varepsilon\to 1$ and $w^\varepsilon_\perp\to 0$ as $\varepsilon\to 0$, we obtain
$$
\sup_{(x,y)\in B(1)}|\bar{P}_\varepsilon(x,y)-\bar{P}_0(x,y)|\to 0\quad\hbox{as}\quad\varepsilon\to 0.
$$
Moreover, in the chart $(\phi_3,U_3)$, $Q_\varepsilon$ and $Q_0$ reads as
$$
\bar{Q}_\varepsilon(x,y)=-\beta y+f_2(x\varphi^\varepsilon(0)+w_\perp^\varepsilon(0,\cdot))+g_2(y),
$$
and 
$$
\bar{Q}_0(x,y)=-\beta y+f_2(x)+g_2(y).
$$
Thus
\begin{align*}
|\bar{Q}_\varepsilon(x,y)-\bar{Q}_0(x,y)|&\leq |f_2(x\varphi^\varepsilon(0)+w_\perp^\varepsilon(0,\cdot))-f_2(x)|\\
&\leq K_2 |x||\varphi^\varepsilon(0)-1|+K_2|w_\perp^\varepsilon(0,\cdot)|.
\end{align*}
By convergence properties $\varphi^\varepsilon\to 1$ and $w^\varepsilon_\perp\to 0$ as $\varepsilon\to 0$, we obtain
$$
\sup_{(x,y)\in B(1)}|\bar{Q}_\varepsilon(x,y)-\bar{Q}_0(x,y)|\to 0\quad\hbox{as}\quad\varepsilon\to 0.
$$

\noindent {\it Step 2. $C^1$-convergence.}

In the chart $(\phi_1,U_1)$, the partial derivatives of $\bar{P}_\varepsilon$ and $\bar{P}_0$ with respect to $x$ reads as
\begin{align*}
\frac{\partial \bar{P}_\varepsilon}{\partial x}(x,y)&=y^{d-1}\lambda_2^\varepsilon-y^{d}\int_0^1f_1\Big(\frac{1}{y}\varphi^\varepsilon(r)+w_\perp^\varepsilon(r,\cdot)\Big)\varphi^\varepsilon(r)\,dr\\
&-y^dg_1\Big(\frac{x}{y}\Big)\varphi^\varepsilon(0)-xy^{d-1}g^\prime_1\Big(\frac{x}{y}\Big)\varphi^\varepsilon(0)-\beta y^{d-1}+y^{d-1}g^\prime_2\Big(\frac{x}{y}\Big)
\end{align*}
and 
\begin{align*}
\frac{\partial \bar{P}_0}{\partial x}(x,y)&=y^{d-1}\lambda-y^{d}f_1\Big(\frac{1}{y}\Big)\\
&-y^dg_1\Big(\frac{x}{y}\Big)-xy^{d-1}g^\prime_1\Big(\frac{x}{y}\Big)-\beta y^{d-1}+y^{d-1}g^\prime_2\Big(\frac{x}{y}\Big).
\end{align*}
Thus,
\begin{align*}
\Big|\frac{\partial \bar{P}_\varepsilon}{\partial x}(x,y)- \frac{\partial \bar{P}_0}{\partial x}(x,y)\Big|&\leq |y^{d-1}|\lambda_2^\varepsilon-\lambda|+|y^{d}|\int_0^1|f_1\Big(\frac{1}{y}\varphi^\varepsilon(r)+w_\perp^\varepsilon(r,\cdot)\Big)\varphi^\varepsilon(r)-f_1\Big(\frac{1}{y}\Big)|\,dr\\
&+|y^dg_1\Big(\frac{x}{y}\Big)||\varphi^\varepsilon(0)-1|+|xy^{d-1}g^\prime_1\Big(\frac{x}{y}\Big)||\varphi^\varepsilon(0)-1|.
\end{align*}

Since $d$ is greater than or equal degree of $g_1$, we have that $|y^{d}g_1(x/y)|$ and $|xy^{d-1}g^\prime_1(x/y)|$ are bounded in $B(1)$ The same boundness is true for $|y^{d}f_1(1/y)|$. By convergence properties $\lambda_2^\varepsilon\to \lambda$, $\varphi^\varepsilon\to 1$ and $w_\perp^\varepsilon\to 0$ as $\varepsilon\to 0$, we obtain
$$
\sup_{(x,y)\in B(1)}\Big|\frac{\partial \bar{P}_\varepsilon}{\partial x}(x,y)- \frac{\partial \bar{P}_0}{\partial x}(x,y)\Big|\to 0\quad\hbox{as}\quad\varepsilon\to 0.
$$

In the chart $(\phi_1,U_1)$, the partial derivatives of $\bar{P}_\varepsilon$ and $\bar{P}_0$ with respect to $y$ reads as
\begin{align*}
\frac{\partial \bar{P}_\varepsilon}{\partial y}(x,y)&=x(d-1)y^{d-2}\lambda_2^\varepsilon-xdy^{d-1}\int_0^1 f_1\Big(\frac{1}{y}\varphi^\varepsilon(r)+w_\perp^\varepsilon(r,\cdot)\Big)\varphi^\varepsilon(r)\,dr\\
&+xy^{d-2}\int_0^1 f^\prime_1\Big(\frac{1}{y}\varphi^\varepsilon(r)+w_\perp^\varepsilon(r,\cdot)\Big)\varphi^\varepsilon(r)\,dr\\
&-xdy^{d-1}g_1\Big(\frac{x}{y}\Big)\varphi^\varepsilon(0)+x^2y^{d-2}g^\prime_1\Big(\frac{x}{y}\Big)\varphi^\varepsilon(0)\\
&-\beta x(d-1)y^{d-2}+dy^{d-1}f_2\Big(\frac{1}{y}\varphi^\varepsilon(0)+w_\perp^\varepsilon(0,\cdot)\Big)-y^{d-2}f^\prime_2\Big(\frac{1}{y}\varphi^\varepsilon(0)+w_\perp^\varepsilon(0,\cdot)\Big)\\
&+dy^{d-1}g_2\Big(\frac{x}{y}\Big)-xy^{d-2}g^\prime_2\Big(\frac{x}{y}\Big)
\end{align*}
and 
\begin{align*}
\frac{\partial \bar{P}_0}{\partial y}(x,y)&=x(d-1)y^{d-2}\lambda-xdy^{d-1}f_1\Big(\frac{1}{y}\Big)\\
&+xy^{d-2}f^\prime_1\Big(\frac{1}{y}\Big)\\
&-xdy^{d-1}g_1\Big(\frac{x}{y}\Big)+x^2dy^{d-2}g^\prime_1\Big(\frac{x}{y}\Big)\\
&-\beta x(d-1)y^{d-2}+dy^{d-1}f_2\Big(\frac{1}{y}\Big)-y^{d-2}f^\prime_2\Big(\frac{1}{y}\Big)\\
&+dy^{d-1}g_2\Big(\frac{x}{y}\Big)-xy^{d-2}g^\prime_2\Big(\frac{x}{y}\Big).
\end{align*}

Since $d-2\geq 0$ and $d-1$ is greater than or equal degree of $g_1$ and $f_1$, as in the previous estimates, we obtain
$$
\sup_{(x,y)\in B(1)}\Big|\frac{\partial \bar{P}_\varepsilon}{\partial y}(x,y)- \frac{\partial \bar{P}_0}{\partial y}(x,y)\Big|\to 0\quad\hbox{as}\quad\varepsilon\to 0.
$$

Moreover, in the chart $(\phi_1,U_1)$, the partial derivatives of $\bar{Q}_\varepsilon$ and $\bar{Q}_0$ with respect to $x$ reads as
\begin{align*}
\frac{\partial\bar{Q}_\varepsilon}{\partial x}(x,y)&=-y^dg^\prime_1\Big(\frac{x}{y}\Big)\varphi^\varepsilon(0) 
\end{align*}
and 
\begin{align*}
\frac{\partial \bar{Q}_0}{\partial x}(x,y)&=-y^dg^\prime_1\Big(\frac{x}{y}\Big).
\end{align*}
Thus,
$$
\sup_{(x,y)\in B(1)}\Big|\frac{\partial \bar{Q}_\varepsilon}{\partial x}(x,y)- \frac{\partial \bar{Q}_0}{\partial x}(x,y)\Big|\to 0\quad\hbox{as}\quad\varepsilon\to 0.
$$

In the chart $(\phi_1,U_1)$, the partial derivatives of $\bar{Q}_\varepsilon$ and $\bar{Q}_0$ with respect to $y$ reads as
\begin{align*}
\frac{\partial\bar{Q}_\varepsilon}{\partial y}(x,y)&=d\lambda_2^\varepsilon y^{d-1}-(d+1)y^d\int_0^1 f_1\Big(\frac{1}{y}\varphi^\varepsilon(r)+w_\perp^\varepsilon(r,\cdot)\Big)\varphi^\varepsilon(r)\,dr\\
&+y^{d-1}\int_0^1 f^\prime_1\Big(\frac{1}{y}\varphi^\varepsilon(r)+w_\perp^\varepsilon(r,\cdot)\Big)\varphi^\varepsilon(r)\,dr\\
&-(d+1)y^dg_1\Big(\frac{x}{y}\Big)\varphi^\varepsilon(0)+xy^{d-1}g^\prime_1\Big(\frac{x}{y}\Big)\varphi^\varepsilon(0) 
\end{align*}
and 
\begin{align*}
\frac{\partial \bar{Q}_0}{\partial y}(x,y)&=d\lambda y^{d-1}-(d+1)y^df_1\Big(\frac{1}{y}\Big)\\
&+y^{d-1} f^\prime_1\Big(\frac{1}{y}\Big)\\
&-(d+1)y^dg_1\Big(\frac{x}{y}\Big)+xy^{d-1}g^\prime_1\Big(\frac{x}{y}\Big). 
\end{align*}
Since $d-1\geq 0$ and $d$ is greater than degree of $g_1$ and $f_1$, as in the previous estimates, we obtain
$$
\sup_{(x,y)\in B(1)}\Big|\frac{\partial \bar{Q}_\varepsilon}{\partial y}(x,y)- \frac{\partial \bar{Q}_0}{\partial y}(x,y)\Big|\to 0\quad\hbox{as}\quad\varepsilon\to 0.
$$

In the chart $(\phi_2,U_2)$, the partial derivatives of $\bar{P}_\varepsilon$ and $\bar{P}_0$ with respect to $x$ reads as
\begin{align*}
\frac{\partial \bar{P}_\varepsilon}{\partial x}(x,y)&=-y^{d-1}\lambda_2^\varepsilon+y^{d-1}\int_0^1f^\prime_1\Big(\frac{x}{y}\varphi^\varepsilon(r)+w_\perp^\varepsilon(r,\cdot)\Big)\varphi^\varepsilon(r)\,dr\\
& +\beta y^{d-1}-y^{d}f_2\Big(\frac{x}{y}\varphi^\varepsilon(0)+w_\perp^\varepsilon(0,\cdot)\Big)-xy^{d-1}f^\prime_2\Big(\frac{x}{y}\varphi^\varepsilon(0)+w_\perp^\varepsilon(0,\cdot)\Big)\\
& -y^dg_2\Big(\frac{1}{y}\Big)
\end{align*}
and 
\begin{align*}
\frac{\partial \bar{P}_0}{\partial x}(x,y)&=-y^{d-1}\lambda+y^{d-1}f^\prime_1\Big(\frac{x}{y}\Big)\\
& +\beta y^{d-1}-y^{d}f_2\Big(\frac{x}{y}\Big)-xy^{d-1}f^\prime_2\Big(\frac{x}{y}\Big)\\
& -y^dg_2\Big(\frac{1}{y}\Big)
\end{align*}

Since $d-1\geq 0$, as in the previous estimates, we obtain
$$
\sup_{(x,y)\in B(1)}\Big|\frac{\partial \bar{P}_\varepsilon}{\partial x}(x,y)- \frac{\partial \bar{P}_0}{\partial x}(x,y)\Big|\to 0\quad\hbox{as}\quad\varepsilon\to 0.
$$

In the chart $(\phi_2,U_2)$, the partial derivatives of $\bar{P}_\varepsilon$ and $\bar{P}_0$ with respect to $y$ reads as
\begin{align*}
\frac{\partial \bar{P}_\varepsilon}{\partial y}(x,y)&=-x(d-1)y^{d-2}\lambda_2^\varepsilon+dy^{d-1}\int_0^1 f_1\Big(\frac{x}{y}\varphi^\varepsilon(r)+w_\perp^\varepsilon(r,\cdot)\Big)\varphi^\varepsilon(r)\,dr\\
&-xy^{d-2}\int_0^1 f^\prime_1\Big(\frac{x}{y}\varphi^\varepsilon(r)+w_\perp^\varepsilon(r,\cdot)\Big)\varphi^\varepsilon(r)\,dr\\
&+dy^{d-1}g_1\Big(\frac{1}{y}\Big)\varphi^\varepsilon(0)-y^{d-2}g^\prime_1\Big(\frac{1}{y}\Big)\varphi^\varepsilon(0)\\
&+\beta x(d-1)y^{d-2}-dxy^{d-1}f_2\Big(\frac{x}{y}\varphi^\varepsilon(0)+w_\perp^\varepsilon(0,\cdot)\Big)-x^2y^{d-2}f^\prime_2\Big(\frac{x}{y}\varphi^\varepsilon(0)+w_\perp^\varepsilon(0,\cdot)\Big)\\
&-dxy^{d-1}g_2\Big(\frac{1}{y}\Big)+xy^{d-2}g^\prime_2\Big(\frac{1}{y}\Big)
\end{align*}
and 
\begin{align*}
\frac{\partial \bar{P}_0}{\partial y}(x,y)&=-(d-1)y^{d-2}\lambda+dy^{d-1}f_1\Big(\frac{x}{y}\Big)\\
&-x^2y^{d-2}f^\prime_1\Big(\frac{x}{y}\Big)\\
&+dy^{d-1}g_1\Big(\frac{1}{y}\Big)-y^{d-2}g^\prime_1\Big(\frac{1}{y}\Big)\\
&+\beta x(d-1)y^{d-2}-xdy^{d-1}f_2\Big(\frac{x}{y}\Big)-x^2y^{d-2}f^\prime_2\Big(\frac{x}{y}\Big)\\
&-xdy^{d-1}g_2\Big(\frac{1}{y}\Big)+xy^{d-2}g^\prime_2\Big(\frac{1}{y}\Big).
\end{align*}

Since $d-2\geq 0$ and $d-1$ is greater than or equal degree of $g_1$ and $f_1$, as in the previous estimates, we obtain
$$
\sup_{(x,y)\in B(1)}\Big|\frac{\partial \bar{P}_\varepsilon}{\partial y}(x,y)- \frac{\partial \bar{P}_0}{\partial y}(x,y)\Big|\to 0\quad\hbox{as}\quad\varepsilon\to 0.
$$

Moreover, in the chart $(\phi_2,U_2)$, the partial derivatives of $\bar{Q}_\varepsilon$ and $\bar{Q}_0$ with respect to $x$ reads as
\begin{align*}
\frac{\partial\bar{Q}_\varepsilon}{\partial x}(x,y)&=xy^{d-1}f^\prime_2\Big(\frac{x}{y}\varphi^\varepsilon(0)+w_\perp^\varepsilon(0,\cdot)\Big)
\end{align*}
and 
\begin{align*}
\frac{\partial \bar{Q}_0}{\partial x}(x,y)&=xy^{d-1}f^\prime_2\Big(\frac{x}{y}\Big).
\end{align*}
Thus,
$$
\sup_{(x,y)\in B(1)}\Big|\frac{\partial \bar{Q}_\varepsilon}{\partial x}(x,y)- \frac{\partial \bar{Q}_0}{\partial x}(x,y)\Big|\to 0\quad\hbox{as}\quad\varepsilon\to 0.
$$

In the chart $(\phi_2,U_2)$, the partial derivatives of $\bar{Q}_\varepsilon$ and $\bar{Q}_0$ with respect to $y$ reads as
\begin{align*}
\frac{\partial\bar{Q}_\varepsilon}{\partial y}(x,y)&=\beta dy^{d-1}-(d+1)y^df_2\Big(\frac{x}{y}\varphi^\varepsilon(0)+w_\perp^\varepsilon(0,\cdot)\Big)+xy^{d-1}f^\prime_2\Big(\frac{x}{y}\varphi^\varepsilon(0)+w_\perp^\varepsilon(0,\cdot)\Big)\\
&-(d+1)y^dg_2\Big(\frac{1}{y}\Big)+y^{d-1}g^\prime_2\Big(\frac{1}{y}\Big)
\end{align*}
and 
\begin{align*}
\frac{\partial \bar{Q}_0}{\partial y}(x,y)&=\beta dy^{d-1}-(d+1)y^df_2\Big(\frac{x}{y}\Big)+xy^{d-1}f^\prime_2\Big(\frac{x}{y}\Big)\\
&-(d+1)y^dg_2\Big(\frac{1}{y}\Big)+y^{d-1}g^\prime_2\Big(\frac{1}{y}\Big).
\end{align*}
Since $d-1\geq 0$, as in the previous estimates, we obtain
$$
\sup_{(x,y)\in B(1)}\Big|\frac{\partial \bar{Q}_\varepsilon}{\partial y}(x,y)- \frac{\partial \bar{Q}_0}{\partial y}(x,y)\Big|\to 0\quad\hbox{as}\quad\varepsilon\to 0.
$$

Finally, In the chart $(\phi_3,U_3)$, the partial derivatives of $\bar{P}_\varepsilon$ and $\bar{P}_0$ with respect to $x$ reads as
\begin{align*}
\frac{\partial \bar{P}_\varepsilon}{\partial x}(x,y)&=-\lambda_2^\varepsilon+\int_0^1f^\prime_1(x\varphi^\varepsilon(r)+w_\perp^\varepsilon(r,\cdot))\varphi^\varepsilon(r)\,dr\\
\end{align*}
and 
\begin{align*}
\frac{\partial \bar{P}_0}{\partial x}(x,y)&=-\lambda +f^\prime_1(x)
\end{align*}
As in the previous estimates, we obtain
$$
\sup_{(x,y)\in B(1)}\Big|\frac{\partial \bar{P}_\varepsilon}{\partial x}(x,y)- \frac{\partial \bar{P}_0}{\partial x}(x,y)\Big|\to 0\quad\hbox{as}\quad\varepsilon\to 0.
$$

In the chart $(\phi_3,U_3)$, the partial derivatives of $\bar{P}_\varepsilon$ and $\bar{P}_0$ with respect to $y$ reads as
\begin{align*}
\frac{\partial \bar{P}_\varepsilon}{\partial y}(x,y)&=g^\prime_1(y)\varphi^\varepsilon(0)
\end{align*}
and 
\begin{align*}
\frac{\partial \bar{P}_0}{\partial y}(x,y)&=g^\prime_1(y)
\end{align*}
As in the previous estimates, we obtain
$$
\sup_{(x,y)\in B(1)}\Big|\frac{\partial \bar{P}_\varepsilon}{\partial y}(x,y)- \frac{\partial \bar{P}_0}{\partial y}(x,y)\Big|\to 0\quad\hbox{as}\quad\varepsilon\to 0.
$$

Moreover, in the chart $(\phi_3,U_3)$, the partial derivatives of $\bar{Q}_\varepsilon$ and $\bar{Q}_0$ with respect to $x$ reads as
\begin{align*}
\frac{\partial\bar{Q}_\varepsilon}{\partial x}(x,y)&=f^\prime_2(x\varphi^\varepsilon(0)+w_\perp^\varepsilon(0,\cdot))
\end{align*}
and 
\begin{align*}
\frac{\partial \bar{Q}_0}{\partial x}(x,y)&=f^\prime_2(x).
\end{align*}
Thus,
$$
\sup_{(x,y)\in B(1)}\Big|\frac{\partial \bar{Q}_\varepsilon}{\partial x}(x,y)- \frac{\partial \bar{Q}_0}{\partial x}(x,y)\Big|\to 0\quad\hbox{as}\quad\varepsilon\to 0.
$$

In the chart $(\phi_3,U_3)$, the partial derivatives of $\bar{Q}_\varepsilon$ and $\bar{Q}_0$ with respect to $y$ reads as
\begin{align*}
\frac{\partial\bar{Q}_\varepsilon}{\partial y}(x,y)&=-\beta+ g^\prime_2(y)
\end{align*}
and 
\begin{align*}
\frac{\partial \bar{Q}_0}{\partial y}(x,y)&=-\beta+ g^\prime_2(y)
\end{align*}
Thus,
$$
\Big|\frac{\partial \bar{Q}_\varepsilon}{\partial y}(x,y)- \frac{\partial \bar{Q}_0}{\partial y}(x,y)\Big|= 0.
$$

Putting all convergence together we obtain \eqref{convC1}.
\end{proof}

\begin{remark}\label{obs}
Notice that, if  $g_1(0)=0$ then we have 
$$
|y^{d-1}g\Big(\frac{x}{y}\Big)||\varphi^\varepsilon(0)-1|\leq K_2|y^d||\varphi^\varepsilon(0)-1|\to 0\quad \hbox{as}\quad \varepsilon\to 0,
$$
where $K_2$ denotes the Lipschitz constant of $g_1$ o $B(1)$. Hence, we can remove the assumption $d-1$ greater than or equal degree of $g_1$ in this case.

A similar remark is valid when $f_1(0)=0$.
\end{remark}

\section{Example: }\label{Examples}

Taking $\lambda =1$, $\beta=1$ $f_1(u)=u-u^3$, $g_1(v)=-v$, , $f_2=-u$ and $g_2=v-v^3$.  By Remark \ref{obs} we have all conditions on $f_1,f_2,g_1$ and $g_2$ satisfied with $d=3$.  The ODE \eqref{unperturbed_equation} takes the form
\begin{equation}\label{unperturbed_equationex1}
\begin{cases}
\dot{u}=-u^3-v,\\
\dot{v}=-v^3-u,
\end{cases}
\end{equation}
which has a saddle point $(0,0)$ and two local attractors $(-1,1)$ and $(1,-1)$.  

To find the equilibrium points at infinity, we analyze $X_0$ in the charts $U_1$ and $U_2$ and set $v=0$.

In the chart $(\varphi_1,U_1)$ the vector field $X_0(u,v)$ has the expression
\begin{equation}
\begin{cases}
\dot{u}=u+u^2v-u^3-v^2,    \\
\dot{v}=v+uv^3,
\end{cases}
\end{equation}
which has a local repeller at $(0,0)$ and two saddle $(1,0)$ and $(-1,0)$ which correspond to equilibrium points at infinity.

In the chart $(\varphi_2,U_2)$ the vector field $X_0(u,v)$ has the expression
$$
\begin{cases}
\dot{u}=-u^3-v^2+u+u^2v^2, \\
\dot{v}=v+uv^3.
\end{cases}
$$
which also has a local repeller at $(0,0)$ and two saddle $(1,0)$ and $(-1,0)$ which correspond to equilibrium points at infinity.

\begin{figure}[h!]
    \centering
    \includegraphics[width=0.32\linewidth]{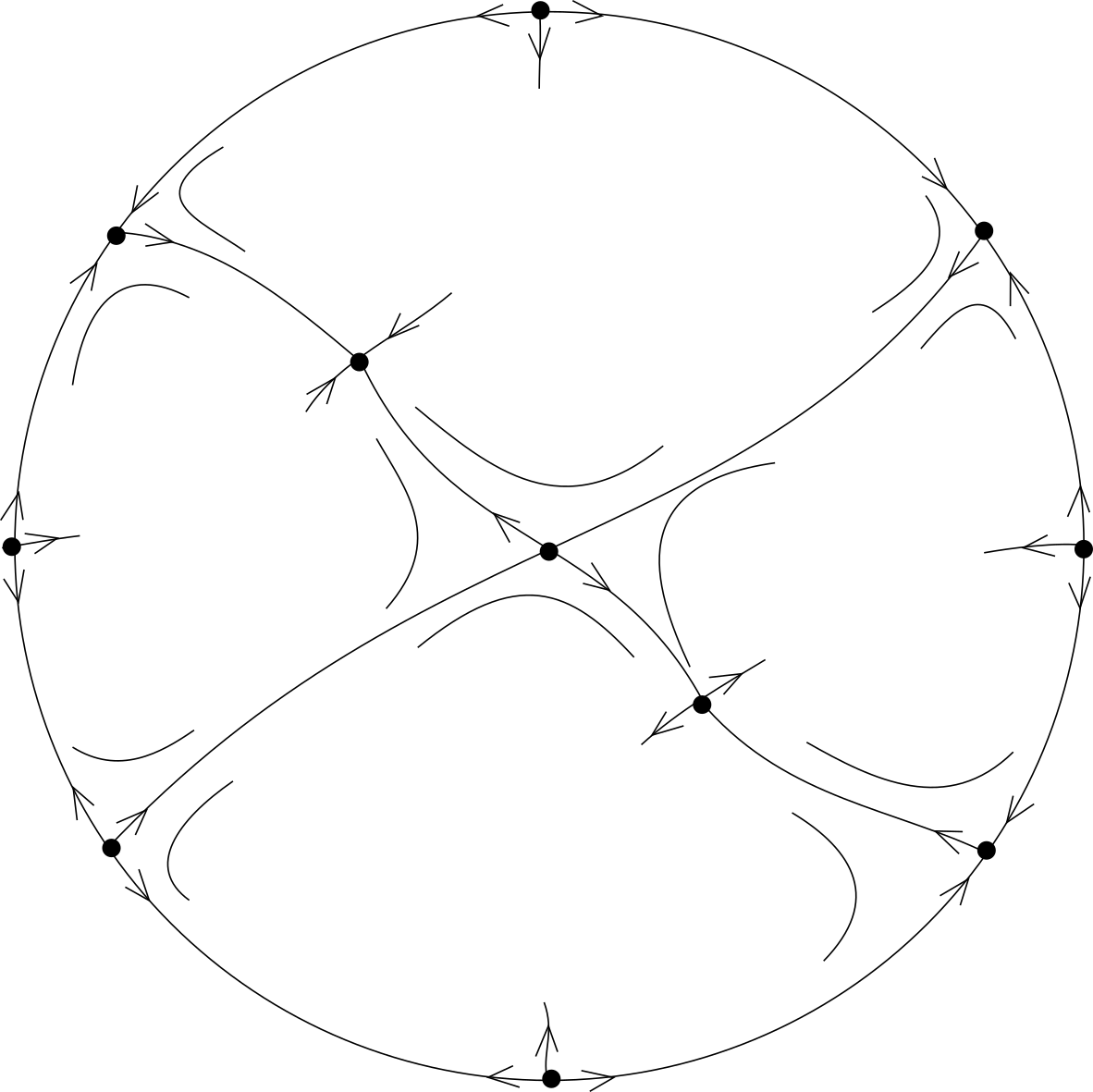}
    \caption{Phase Portrait}
    \label{}
\end{figure}

\bibliographystyle{abbrv}
\bibliography{References}
\end{document}